\documentclass[11pt]{amsart}
\usepackage{graphicx,amssymb,amsmath, latexsym,cite, amscd,stmaryrd}
\usepackage{amsthm,amsrefs}
\usepackage{enumerate,lscape}  
\usepackage{float,xcolor}
\usepackage{fancyhdr, lastpage}  
\usepackage{libertine}         
 \usepackage{url}   
\usepackage{verbatim} 
\usepackage{algorithm}  
\usepackage[noend]{algpseudocode}
  
\algblock[TryCatchFinally]{try}{endtry}
\algcblock[TryCatchFinally]{TryCatchFinally}{finally}{endtry} 
\algcblockdefx[TryCatchFinally]{TryCatchFinally}{catch}{endtry}
	[1]{\textbf{catch} #1}
	{\textbf{end try}}

\usepackage[all,cmtip]{xy}

\usepackage{algorithm}
\usepackage[noend]{algpseudocode}

\algblock[TryCatchFinally]{try}{endtry}
\algcblock[TryCatchFinally]{TryCatchFinally}{finally}{endtry}
\algcblockdefx[TryCatchFinally]{TryCatchFinally}{catch}{endtry}
	[1]{\textbf{exception:} #1}
	{\textbf{end try}}

\newtheoremstyle{mythm}                   
{6pt}
{6pt}
{\it}
{}
{\bf}
{.}
{.5em}
{}

\newtheoremstyle{mydef}                   
{6pt}
{6pt}
{}
{}
{\bf}
{.}
{.5em}
{}

\newtheoremstyle{myrem}                   
{6pt}
{6pt}
{}
{}
{\bf}
{.}
{.5em}
{}

\theoremstyle{mythm}      
\newtheorem{theorem}{Theorem}[section]
\newtheorem{proposition}[theorem]{Proposition}
\newtheorem{lemma}[theorem]{Lemma}

\theoremstyle{mydef}      
\newtheorem{definition}[theorem]{Definition}

\theoremstyle{myrem}
\newtheorem{remark}[theorem]{Remark}
\numberwithin{equation}{section}

\newcounter{ithmcount}

\newenvironment{ithm}{\begin{list}{{\rm \alph{ithmcount})}}{\usecounter{ithmcount}\labelwidth18pt
      \leftmargin18pt \topsep3pt \itemsep1pt \parsep2pt}}{\end{list}}

\newenvironment{items}{
\begin{list}{$\alph{item})$}
{\labelwidth18pt \leftmargin20pt \topsep3pt \itemsep5pt \parsep0pt}}
{\end{list}}

\allowdisplaybreaks

\reversemarginpar
\subjclass[2000]{}

\newcommand{\Aut}{{\rm Aut}}

\newcommand{\GL}{{\rm GL}}  
\newcommand{\Soc}{{\rm soc}}

\renewcommand{\leq}{\leqslant}

\begin{document}

\vspace*{-1.4cm}
   
\title{Isomorphism testing of groups of cube-free order}
\subjclass[2000]{}
\author[H. Dietrich]{Heiko Dietrich}
\address{School of Mathematics, Monash University, 
Clayton VIC 3800, Australia}
\email{Heiko,Dietrich@Monash.Edu}
\author[J.B. Wilson]{James B.\ Wilson}
\address{Department of Mathematics, Colorado Sate University, Fort Collins Colorado, 80523, USA}
\email{James.Wilson@ColoState.Edu}

\date{\today}

\begin{abstract} A group $G$ has cube-free order if no prime to the third power divides $|G|$. We describe an algorithm that given two cube-free groups $G$ and $H$ of known order, decides whether $G\cong H$, and, if so, constructs an isomorphism $G\to H$. If the groups are input as permutation groups, then our algorithm runs in time polynomial in the input size, improving on the previous super-polynomial bound. An implementation of our algorithm is provided for the computer algebra system {\sf GAP}. \bigskip
  
\begin{center}
  \small \emph{In memory of C.C.\ Sims.}
\end{center}
\end{abstract}

\thanks{This research was partially supported by the Simons Foundation, the Mathematisches Forschungsinstitut Oberwolfach, and  NSF grant DMS-1620454. 
The first author thanks the Lehrstuhl D of the RWTH Aachen for the great hospitality during his Simon Visiting Professorship in Summer 2016. Both authors thank Alexander Hulpke for advice on conjugacy problems.}

\maketitle


\vspace*{-0.6cm}

\section{Introduction}

Capturing the natural concept of symmetry, groups are one the most prominent algebraic 
structures in science. Yet, it is still a challenge to decide whether two finite groups are isomorphic. Despite abundant knowledge 
about groups, presently no one has provided an isomorphism test for all finite groups 
whose complexity  improves substantively over brute-force.  In the most general form
there is no known polynomial-time isomorphism test even for non-deterministic Turing machines, that is,
the problem may lie outside the complexity classes NP and co-NP (see \cite{BS84}*{Corollary~4.9}).
At the time of this writing, the available implementations of algorithms that test 
isomorphism on broad classes of groups can run out of memory or run for days on examples 
of orders only a few thousand, see \cite{BMW}*{Section~1.1} and Table~1.   To 
isolate the critical difficulties in group isomorphism, it helps to consider special 
classes of groups as has been done recently in 
{\citelist{\cite{BQ:tower}\cite{BCGQ}\cite{BMW}\cite{cf}\cite{PART1}\cite{Wilson:profile}}}.

{This paper is a part of a larger project intended to describe for which orders of groups is group isomorphism tractable: details of this project are given in  \cite{PART1}. In particular, in \cite{PART1} we have described} polynomial-time algorithms for isomorphism testing of  abelian and meta-cyclic groups of \emph{most} orders; the computational framework for these algorithms is built upon type theory and groups of so-called black-box type. By a theorem of H\"older (\cite{rob}*{10.1.10}), all groups of square-free order are coprime meta-cyclic, that is, they can be decomposed as $G=A\ltimes B$ where $A,B\leq G$ are cyclic subgroups of coprime orders; unfortunately, {\cite[Theorem 1.2]{PART1}} is not guaranteed for all square-free orders. In this paper, we switch to a more restrictive computational model, allowing us to make progress for isomorphism
testing of square-free and cube-free groups.  Specifically, here  we consider groups generated by a set $S$ of permutations on a finite set $\Omega$.  That  gives us access to a robust family of 
algorithms by Sims and many others (see \citelist{\cite{handbook}\cite{Seress}})
that run in time polynomial in $|\Omega|\cdot |S|$. Note that the order of such a group 
$G$ can be exponential in $|\Omega|\cdot |S|$, even when restricted to groups of square-free order, see Proposition~\ref{prop:degree}. The main result of this paper is the following theorem.
 
\begin{theorem}\label{thm:main} 
There is an algorithm that given groups $G$ and $H$ of permutations on
finitely many points, decides whether they are of cube-free order, and if so, decides that $G\not\cong H$ or constructs an isomorphism $G\to H$. The algorithm runs in time polynomial in the input size.  
\end{theorem}

Theorem \ref{thm:main} is based on the structure analysis of cube-free groups  by Eick \& Dietrich \cite{cf} and  Qiao \& Li \cite{cf2}. A top-level description of our algorithm is given in Section \ref{sec:summary}. Importantly, our algorithm translates to a functioning implementation for the system \textsf{GAP} \cite{gap}, in the package ``Cubefree''~\cite{cube-free}. As a side-product, we also discuss algorithms related to the  construction of complements of $\Omega$-groups, Sylow towers, socles, and constructive presentations, see Section \ref{sec:prelres}. These algorithms have applications beyond cube-free groups and might be of general interest in computational group theory.

\subsection{Limitations}
In contrast to our work in {\cite{PART1}},  Theorem \ref{thm:main} no longer applies to a dense set 
of orders: the density of positive integers $n$ which are square-free and  cube-free tends 
to $1/\zeta(2)\approx 0.61$ and  $1/\zeta(3)\approx 0.83$, respectively,
where $\zeta(x)$ is the Riemann $\zeta$-function, see \cite{num}*{(2)}. It is known that most isomorphism types of groups 
accumulate at orders with large prime-power divisors.  Indeed,  Higman, Sims, and Pyber 
\cite{BNV:enum} proved that the number of groups of order $n$, up to isomorphism, tends to 
$n^{2\mu(n)^2/27+O(\log n)}$ where $\mu(n)=\max\{k : n \text{ is not $k$-free}\}$. 
Specifically, the number of pairwise non-isomorphic groups of a cube-free order $n$ is 
not more than $O(n^8)$, with speculation that the tight bound is $o(n^2)$, see \cite{BNV:enum}*{p.~236}. 
The prevailing belief in works like \citelist{\cite{BCGQ}\cite{Wilson:profile}}
is that the difficult instances of group isomorphism
are when $\mu(n)$ is unbounded, especially when $n$ is a prime power. 
Isomorphism testing of finite $p$-groups is indeed a research area that has attracted a lot of attention.

However, 
Theorem \ref{thm:main} completely handles an easily described family of group orders which 
may make it easier to use in applications.  A further point is that groups of cube-free order 
exhibit many of the fundamental components of finite groups.  For instance, groups of cube-free
order need not be solvable, to wit the simple alternating group $A_5$ has cube-free order $60$.  When decomposed into canonical 
series, such as the Fitting series, the associated extensions have nontrivial 
first and second cohomology groups -- a measure of how difficult it is to compare different extensions.

\subsection{Structure of the paper}
In Section \ref{sec:prelim} we introduce some notation and comment on the computational model for our algorithm. In Section \ref{secStructure} we recall the structure of cube-free groups and give a top-level description of our isomorphism test. Various preliminary algorithms (for example, related to the construction of Sylow bases and towers, $\Omega$-complements, socles, and constructive presentations) are described in Section~\ref{sec:prelres}. The proof of the main theorem is broken up into three progressively more general families: the solvable Frattini-free case (Section~\ref{secsolvFF}), the general solvable case (Section~\ref{sec:solv}), and finally the general case (Section~\ref{sec:general}).   We have implemented many aspects of this algorithm in the computer algebra system {\sf GAP} and comment on some examples in Section~\ref{sec:ex}.

\section{Notation and computational model}\label{sec:prelim}

\subsection{Notation.} We reserve $p$ for prime numbers and $n$ for group orders.  
For a positive integer $n$ we denote by $C_n$ a
cyclic group of order $n$, and $\mathbb{Z}/n$ for the explicit encoding as 
integers, in which we are further permitted to treat the structure as a
ring.  Let $(\mathbb{Z}/n)^{\times}$ denote the units of this ring.
Direct products of groups are denoted variously by
``$\times$'' or exponents. Throughout, $\mathbb{F}_{q}$ is a field of order $q$ 
and ${\rm GL}_d(q)$ is the group of invertible $(d\times d)$-
matrices over $\mathbb{F}_{q}$. The group ${\rm PSL}_d(q)$ 
consists of matrices of determinant $1$ modulo scalar matrices.

For a group $G$ and  $g,h\in G$, conjugates and commutators are $g^h=h^{-1}g h$ and 
$[g,h]=g^{-1} g^h$, respectively. For subsets $X,Y\subset G$ let
$[X,Y]=\langle [x,y]:x\in X, y\in Y\rangle$; the centralizer and normalizer of $X$ in $G$ are  $C_G(X)=\{g\in G: [X,g]=1\}$ and $N_G(X)=\{g\in G: [X,g]\subseteq X\}$, respectively.
The derived series of $G$ has terms  $G^{(n+1)}=[G^{(n)},G^{(n)}]$ for 
$n\geq 1$, with $G^{(1)}=G$. We read group extensions from the right and use 
$A\ltimes B$ for split extensions; we also write 
$A\ltimes_\varphi B$ to emphasize the action $\varphi\colon A\to \Aut(B)$. 
Hence, $A\ltimes B\ltimes C\ltimes D$ stands for 
$((A\ltimes B)\ltimes C)\ltimes D$, etc.

We mostly adhere to protocol set out in standard literature on computational group theory, such as the Handbook of Computation Group Theory \cite{handbook} and the books of Robinson \cite{rob} and Seress \cite{Seress}.

\subsection{Computational model}\label{secCF}
Throughout  we assume that groups are given as finite permutation groups, but it  is permissible 
to include  congruences, which are best described as quotients of permutation groups. This allows us to prove that the algorithm of Theorem \ref{thm:main} runs in polynomial time in the input size. Proving the same for  groups given by polycyclic presentations seems difficult, partly because of the challenges involving {\it collection}, see~\cite{collection}. {\bfseries\itshape Convention:} when we say that an algorithm runs in polynomial time, then this is to be understood to be in time polynomial in the input size, assuming that the groups are input as  (quotients of) finite permutation groups.

One 
simple but critical implication of our computational model  is that if a prime $p$ divides the group 
order $|G|$, then $p$ divides~$d!$, where $d$ is the size of the permutation domain; so 
$p\leq d$, which is less than the input size for $G$.  This shows that all primes dividing 
the group order are \emph{small}, allowing for polynomial-time factorization and other 
relevant number theory. Moreover, many essential group theoretic structures of groups of 
permutations (and their quotients) can be computed in polynomial time, as outlined in 
\cite{Seress}*{p.\ 49} and \cite{KL:quo}*{Section~4}.  For example, it is
possible to compute group orders, to produce constructive presentations, and 
to test membership constructively. For solvable permutation groups one can also efficiently 
get  a constructive polycyclic presentation (see Lemma~\ref{lem:permpc}). 

Before we begin, we demonstrate that the assumption that our groups are input by permutations
is not an automatic improvement in the complexity.  In particular, we
show that large groups of square-free (and so also cube-free) order can arise 
as permutation groups in small degrees. 

\begin{proposition}\label{prop:degree} 
Let $G$ be a square-free group of order $n=p_1\cdots p_{\ell}$, with each $p_i$ prime.
The group $G$ can be faithfully represented in a permutation group of degree $p_1+\cdots+p_{\ell}$.
For infinitely many square-free $m$, there is a faithful
permutation representation of the groups of order $m$ on $O(\log^2 m)$~points. 
\end{proposition} 
\begin{proof} 
H\"older's classification \cite{rob}*{(10.1.10)} shows that $G\cong C_a\ltimes C_b$ 
with $n=ab$. Since $a$ is square-free, all subgroups of $C_a$ are direct factors, thus
$C_a=C_d\times C_{C_a}(C_b)$ for a subgroup $C_d$, and $C_a\ltimes C_b=C_d\ltimes C_e$ 
where the centralizer in $C_d$ of $C_e$ is trivial. Thus, 
we can assume  that $C_a\ltimes C_b$ with $C_a$ acting faithfully on $C_b$, and 
$a=p_{1}\cdots p_s$ and $b=p_{s+1}\cdots p_{\ell}$.  Using disjoint $p_i$-cycles for 
each $i>s$, we faithfully represent $C_b$ on  $p_{s+1}+\cdots +p_{\ell}$ points.  
Since $C_a$ acts faithfully on $C_b$, that representation can be given on the disjoint 
cycles of $C_b$, that is, $C_a\ltimes C_b$ is faithfully represented on $p_{s+1}+\cdots +p_{\ell}$ 
points. The first claim follows. For the last observation, let $m=r_1\cdots r_\ell$ be the 
product of the first   $\ell$-primes.  These primorials have asymptotic growth 
$m\in \exp((1+\Theta(1))\ell\log \ell)$, see \cite{euler}*{(3.16)}.  Meanwhile, as just shown, 
the groups of order $m$ can all be represented faithfully on as few as $r_1+\cdots +r_{\ell}$ 
points, and $r_1+\ldots+r_\ell\in \Omega(\ell^2\log \ell)$ by  \cite{MR:primes}*{Theorem~C}.  
\end{proof}

\section{Summary of the algorithm}\label{secStructure}

\subsection{Structure of cube-free groups} For a finite 
group~$G$ we denote by $\Phi(G)$ and $\Soc(G)$ its  {\em Frattini} subgroup and 
its {\em socle}, respectively; the first is the intersection of all maximal 
subgroups of $G$, and the latter is  the subgroup generated by all minimal 
normal subgroups. We write $G_\Phi$ for the Frattini quotient $G/\Phi(G)$. A group is {\em Frattini-free} if $\Phi(G)=1$; in particular, 
$G_\Phi$ is Frattini-free. By \cite{cf}, every group $G$ of cube-free order can be decomposed as
\[G=A\times L\]
where $A$ is trivial or $A=\text{PSL}_2(p)$ for a prime $p>3$ with $p\pm 1$ cube-free, and $L$ is solvable with abelian Frattini subgroup $\Phi(L)=\Phi(G)$ whose order is square-free and divides the order of the Frattini quotient $L_\Phi=L/\Phi(L)$. The latter satisfies 
\[L_\Phi = K\ltimes (B\times C)\]
where $\Soc(L_\Phi)=B\times C$ is the socle of $L_\Phi$ with 
\[B=\prod\nolimits_{i=1}^s \mathbb{Z}/{p_i}\quad\text{and}\quad C=\prod\nolimits_{j=s+1}^m  (\mathbb{Z}/{p_j})^2\]
for distinct primes $p_1,\ldots,p_m$.  Let $X\Yup Y$ denote a \emph{subdirect product}, that is,  a subgroup of $X\times Y$
whose projections to $X$ and $Y$ are surjective.  With this notation, we have 
\[K=K_1\Yup\ldots\Yup K_m \leq \Aut(B\times C) =\prod\nolimits_{i=1}^s \GL_1(p_i)\times \prod\nolimits_{j=s+1}^m \GL_2(p_j);\]
 It follows from work of Gasch\"utz (see \cite[Lemma 9]{cf}) that two solvable Frattini-free groups $K\ltimes (B\times C)$ and $\tilde K\ltimes (B\times C)$ with $K,\tilde K\leq\Aut(B\times C)$ as above are isomorphic if and only if $K$ and $\tilde K$ are conjugate in $\Aut(B\times C)$; this is one of the reasons why our proposed isomorphism algorithm works so efficiently.  Lastly, we recall that $L$ is determined by $L_\Phi$: there exists, up to isomorphism, a unique extension $M$ of $L_\Phi$ by $\Phi(L)$ such that $\Phi(M)\cong \Phi(L)$ and $M/\Phi(M)\cong L_\Phi$, see \cite[Theorem 11]{cf}.

 \begin{remark}\label{remMeta}
 Taunt \cite{taunt} was probably the first who considered the class of cube-free groups. 
The focus in the work of Dietrich \& Eick  \cite{cf} is on a construction algorithm for all cube-free groups of a fixed order, 
up to isomorphism; the approach is based on the so-called \emph{Frattini extension method} 
(see \cite{handbook}*{\S 11.4.1}). Complimentary to this work,  Qiao \& Li \cite{cf2} also analyzed 
the structure of cube-free groups. They proved in \cite{cf2}*{Theorem~1.1} that for every group $G$ 
of cube-free order there exist integers $a,b,c,d>0$ such that $G$ is isomorphic to
\[\begin{array}{ll}
  (C_c\times C_d^2)\ltimes (C_a\times C_b^2)&\text{ or}\\[0.5ex]
G_2\ltimes (C_c\times C_d^2)\ltimes (C_a\times C_b^2)&\text{ with $G_2\leq G$ a Sylow 2-subgroup, or}\\[0.5ex]
\text{PSL}_2(p)\times (C_c\times C_d^2)\ltimes (C_a\times C_b^2)&\text{ for some prime $p$.}
\end{array}\]
Left unclassified in this description are the relevant actions of the semidirect products, and a 
classification up to isomorphism.  As we have shown  in {\cite{PART1}*{Section 4}}, even for meta-cyclic groups, recovering the 
appropriate actions and comparing them is in general not easy.
\end{remark}

Among the implications of these decomposition results is that a solvable group $G$ of 
cube-free order has a Sylow tower, that is, a normal series such that each section is 
isomorphic to a Sylow subgroup of $G$, cf.\ \cite{cf2}*{Corollary~3.4 \& Theorem~3.9}.

\subsection{The algorithm}\label{sec:summary}

Let $G$ and $\tilde G$ be cube-free groups. We now describe the main steps of our algorithm to  construct an isomorphism  $G\to \tilde G$, which fails if and only if $G\not\cong\tilde G$. Our approach is to determine, for each group, the Frattini extension structure as described in Section \ref{secStructure}.  Since our groups are input by permutations, it is possible to 
decide if $|G|=|\tilde{G}|$  and also to factorize this order.  It simplifies 
our treatment to assume that the groups are of the same order and that the prime factors of
this order are known. First, for $G$ (and similarly for $\tilde G$) we do the following:

\smallskip

\begin{enumerate}
\item[(i)] Decompose $G=A\times L$ with $A=1$ or $A=\text{PSL}_2(p)$ simple, and $L$ solvable.
\item[(ii)] Compute the Frattini subgroup $\Phi(L)$ and the Frattini quotient $L_\Phi=L/\Phi(L)$.
\item[(iii)] Compute $\Soc(L_\Phi)=B\times C$ and $K\leq\Aut(B\times C)$ such that $L_\Phi=K\ltimes (B\times C)$.
\end{enumerate}

\smallskip

Then we proceed as follows; if one of these steps fails, then $G\not\cong \tilde G$ is established:

\smallskip

\begin{enumerate}
\item Construct an isomorphism $\psi_A\colon A\to \tilde A$.
\item Construct an isomorphism $\psi_\Phi\colon L_\Phi\to \tilde L_\Phi$.
\item Extend $\psi_\Phi$ to an isomorphism $\psi_L\colon L\to \tilde L$.
\item Combine $\psi_A$ and $\psi_L$ to  an isomorphism $\psi\colon G\to \tilde G$. 
\end{enumerate}

\medskip

In fact, $G$ and $\tilde G$ are isomorphic if and only if we succeed in Steps (1) \& (2). Thus, if we just want to decide whether $G\cong \tilde G$, then Steps (3) \& (4) need not to be carried out; moreover, it is not necessary to construct $\psi_A$: since $A$ and $\tilde A$ are groups of type $\text{PSL}_2$, we have $A\cong \tilde A$  if and only if $|A|=|\tilde A|$, which can be readily determined in our computational framework.


\section{Preliminary algorithms}\label{sec:prelres}

We list a few  algorithms which are required later. One important result is the description of 
an algorithm to construct an abelian Sylow tower for a solvable group, if it exists.  This is a key ingredient in \cite{BQ:tower}, but in that work groups are input as  multiplication tables; in our setting multiplication tables might be exponentially larger than the input, so we cannot use this work. 
  
\subsection{Constructive presentations and $\Omega$-complements}
Let $\Omega$ be a set. An \emph{$\Omega$-group} is a group $G$ on which the set $\Omega$ acts via a prescribed map $\theta\colon \Omega\to \Aut(G)$. We first investigate the problem {\sf $\Omega$-ComplementAbelian}: given an abelian normal $\Omega$-subgroup $M\leq G$, decide whether $G=K\ltimes M$ for some $\Omega$-subgroup $K\leq G$, or certify that no such $K$ exists. Variations on this problem have been discussed in several places; the version we describe is based on a proof in \cite{Wilson:RemakII}*{Proposition~4.5} which extends independent proofs by Luks and Wright in lectures
at the U.\ Oregon.

We show in Proposition \ref{prop:complement}  that {\sf $\Omega$-ComplementAbelian} has a  polynomial time solution for solvable groups. The proof involves Luks' \emph{constructive presentations} \cite{Luks:mat}*{Section 4.2}, which will also be useful later to equip solvable permutation groups with  polycyclic presentations, see Lemma \ref{lem:permpc}.
 
\begin{definition}\label{def:pres}
Let $G$ be a group and $N\unlhd G$. A {\em constructive presentation} of a group $G/N$ 
is a free group $F_X$ on a set $X$, a homomorphism $\phi\colon F_X\to G$, a function 
$\psi\colon G\to F_X$, and a set $R\subset F_X$ such that $g^{-1}(g\psi\phi)\in N$ for every $g\in G$, and  $N\phi^{-1}=\langle R^{F_X}\rangle$, the normal closure of $\langle R\rangle$ in $F_X$.
\end{definition}
This can be interpreted as follows: $\langle X\mid R\rangle$ is a generator-relator presentation of the group 
$G/N$, see \cite{Luks:mat}*{Lemma 4.1}; the homomorphism $\phi$ is defined by assigning  the generators $X$ of $F_X$  
to the generating set $S\subset G$. The function $\psi$ is in general not a homomorphism, and serves to writes elements of $G$ as a corresponding word in $X$. The next lemma  discusses a constructive presentation for a subgroup of the holomorph $\Aut(G)\ltimes G$ of a group $G$.

\begin{lemma}\label{lem:pres} 
  Let $G$ be an $\Omega$-group via $\theta\colon \Omega\to \Aut(G)$, and write $g^w=g^{w\theta}$ for $g\in G$ and $w\in \Omega$. Let $\langle X\mid R\rangle$ with $\phi\colon F_X\to G$ and $\psi\colon G\to F_X$ be a constructive presentation of $G$. Let $\langle \Omega\mid S\rangle$ be a presentation for $A=\langle \Omega\theta\rangle\leq\Aut(G)$. Then
$\langle \Omega\sqcup {X}\mid  S\ltimes{R}\rangle$
is a presentation for $A\ltimes G$ 
where
\begin{eqnarray*}
  S\ltimes{R}&=& S\sqcup{R}\sqcup \{(x\phi)^{w}\psi\cdot (x^{w})^{-1} : x\in{X},w\in\Omega\}\subset F_{\Omega\sqcup{X}}
\end{eqnarray*}
with embedding 	$\theta\sqcup \phi\colon \Omega\sqcup {X}\to A\ltimes G$, $z\mapsto \left\{\begin{array}{cc}z\theta & (z\in\Omega)\\z\phi & (z\in {X})	\end{array}\right.$.
\end{lemma}
\begin{proof}Without loss of generality, we can assume that $F_X=\langle X\rangle$, $F_\Omega=\langle \Omega\rangle$, and 
 $F_\Omega,F_X\leq F_{\Omega\sqcup{X}}$.  Let $K$ be the normal closure of $S\ltimes {R}$ in 
 $F_{\Omega\sqcup{X}}$. Recall that, by definition, if  $x\in X$, then  $x\phi\psi$ and $x$ define the same element in $G$ via $\phi$. It follows that if $w\in\Omega$ and $x\in X$, then $x^w,(x\phi)^w\psi\in F_{\Omega\sqcup X}$ define the same element in $A\ltimes G$ via $\theta\sqcup \phi$: if  $\alpha\colon F_{\Omega\sqcup X}\to A\ltimes  G$ is the homomorphism defined by $\theta\sqcup\phi$, then
 \[ (x^{w})\alpha
		 = ((w\theta)^{-1},1)(1,x\phi)(w\theta,1) 
		 = (1,(x\phi)^{w})
		 = (1, (x\phi)^{w} \psi \phi)
		 = ((x\phi)^{w} \psi) \alpha\]
                 shows that  $(x\phi)^{w}\psi (x^{w})^{-1}\in \ker \alpha$, so $K\leq \ker \alpha$. Now consider $N=KF_X$. From what is said above, if $w\in\Omega$ and $x\in{X}$, then  $Kx^{w}=K(x\phi)^{w}\psi\leq N$, so $N^w=K^w F_X^w\leq \langle Kx^w : x\in {X}\rangle=N$. This shows that  $N\unlhd F_{\Omega\sqcup{X}}$; note that $K^w=K$ since $K$ is the normal closure in $F_{\Omega\sqcup X}$. Now set  
$C=KF_\Omega$.  It follows that $F_{\Omega\sqcup{X}}
=CN$,  thus $H=F_{\Omega\sqcup{X}}/K=CN/K=(C/K)(N/K)$ 
and $N/K$ is normal in $H$.  
Since $C/K$ and $N/K$ satisfy the presentations for $A$ and $G$ respectively, von Dyck's Theorem \cite{rob}*{(2.2.1)} implies that  $H$ is a quotient of $A\ltimes G$. To show that $H$ is isomorphic to $A\ltimes G$ 
it suffices to notice that $A\ltimes G$ satisfies the relations in 
$S\ltimes{R} $ with respect to
$\Omega\sqcup{X}$ and $\theta\sqcup\phi$. As shown above,  $K\leq \ker \alpha$. Since $H=F_{\Omega\sqcup {X}}/K$ is a quotient of the group
$A\ltimes G=F_{\Omega\sqcup {X}}\alpha$, it follows that $K=\ker \alpha$, and therefore  $\langle \Omega\sqcup {X}\mid  S\ltimes{R} \rangle$ is a presentation for $A\ltimes G$. 
\end{proof}

We now show that {\sf $\Omega$-ComplementAbelian} has a polynomial-time solution for solvable groups.

\begin{proposition}\label{prop:complement}
Let $G$ be a solvable $\Omega$-group with abelian normal $\Omega$-subgroup $M\leq G$. There is a polynomial time algorithm that decides whether $G=K\ltimes M$ for some $\Omega$-subgroup $K$, or certifies that no such $K$ exists. 
\end{proposition}
\begin{proof}
Let $G$ be a quotient of a permutation group on $n$ letters, let $\theta\colon \Omega\to \Aut(G)$ be a function, 
and let $M$ be an abelian $(\Omega\cup G)$-subgroup of $G$. We first describe the algorithm, then prove correctness. We use the algorithm of {\cite{PART1}*{Lemma 4.11}} to produce a constructive 
presentation for the solvable quotient $G/M$ with data $\langle X\mid R\rangle$ and maps $\phi\colon X\to G$ and $\psi\colon G\to F_X$. For each    
$s\in \Omega$ and $x\in{X}$, define
\begin{align*}
	w_{s,x} & =((x\phi)^s)\psi\cdot(x^{s})^{-1}\in F_{\Omega\sqcup{X}}.
\end{align*}
Let $\nu\colon X\to M\leq G$ be a function. Considering each $w\in F_{\Omega\sqcup{X}}$ as a word in $\Omega\sqcup{X}$, we denote by $w(\phi\nu)$ the element in $G$ where each symbol $x\in\Omega\sqcup{X}$ in $w$ has been replaced by $(x\phi)(x\nu)$. Use  {\sf  Solve} \cite{KLM}*{Section~3.2} to decide if there is a a function $\nu\colon X\to M$, where
\begin{align}\label{eq:comp-rel-1}
	\forall w\in {R}:&\quad w(\phi\nu) = 1,\textnormal{ and }\\
\label{eq:comp-rel-2}
	\forall s\in\Omega,\forall x\in{X}: & \quad w_{s,x}(\phi\nu)=1.
\end{align}
If no such $\nu$ exists, then report that $M$ has no 
$\Omega$-complement; otherwise, return the group \[K=\langle (x\phi)(x\nu) : x\in{X}\rangle.\]
We show that this  is correct. Let $A=\langle\Omega\theta\rangle\leq \Aut(G)$ and let 
$\langle \Omega\mid {R} '\rangle$ be a presentation of $A$ with
respect to $\theta$. Lemma~\ref{lem:pres} shows that  $\langle\Omega\sqcup {X}\mid {R} '\ltimes{R} \rangle$
is a presentation for $A\ltimes(G/M)$ with respect to $\theta\sqcup \phi$; note that we need not to compute ${R}'$. 

First suppose that the algorithm returns $K=\langle (x\phi) (x\nu):x\in{X}\rangle$.  
As $\{x\phi: x\in {X}\}\subseteq KM$  we get that 
$G=\langle x\phi: x\in {X}\rangle\leq KM\leq G$.  
Since 
$w(\phi\nu)=1$ for all $w\in {R} $ by  \eqref{eq:comp-rel-1}, the group $K$ satisfies the defining
relations of $G/M\cong K/(K\cap M)$, which forces 
$K\cap M=1$, and so $G=K\ltimes M$.  
By \eqref{eq:comp-rel-1} and \eqref{eq:comp-rel-2}, the generator set 
$\Omega\theta\sqcup\{(x\phi) (x\nu):x\in{X} \}$ of 
$\langle A,K\rangle$ satisfies the defining relations ${R} '\ltimes {R} $
of $(A\ltimes G)/M$, and so $\langle A,K\rangle$ is isomorphic to a quotient of 
$(A\ltimes G)/M$ where $K$ is the image of $G/M$.  This shows that $K$ is normal in  
$\langle A,K\rangle$, in particular, $\langle K^{\Omega}\rangle\leq K$.  This proves that if the algorithm 
returns a subgroup, then the output is correct.

Conversely, suppose  $G=K\ltimes M$ such that $K^{\Omega}\subset K$ and there is an idempotent endomorphism $\tau\colon G\to G$ with kernel $M$ and image $K$.
We must show that in this case equations \eqref{eq:comp-rel-1} and \eqref{eq:comp-rel-2} have a solution, so that the algorithm returns a complementary $\Omega$-subgroup to $M$. 
Define the map $\nu\colon X\to M$ by $x\nu=(x\phi)^{-1} (x\phi\tau)$. Now   $K=G\tau=\langle (x\phi) (x\nu): x\in{X} \rangle$ is isomorphic to $G/M$ via  $(x\phi) (x\nu)\mapsto x\phi M$, hence  $\{(x\phi) (x\nu) : x\in {X}\}$ satisfies the relations~${R}$.  Moreover, we have $K^{\Omega}\subseteq K$, so the isomorphism $K\cong G/M$  defined by $(x\phi) (x\nu)\mapsto x\phi M$ extends to $A\ltimes K\to A\ltimes (G/M)$; thus, for all $s\in \Omega$ and $x\in {X}$ we have $w_{s,x}(\phi\nu)=1$. The claim on the complexity follows since we only applied polynomial-time algorithms.
\end{proof}

We will also need to find direct complements; we follow the algorithm in  \cite{Wilson:RemakII}*{Theorem~4.8}. The analysis has not appeared in print so we include its proof.

\begin{proposition}\label{prop:direct}
Let $G$ be an $\Omega$-group and let $U,V\leq G$ be  normal $\Omega$-subgroups with $U\leq V$.  There is a polynomial time algorithm which decides whether $V/U$ is a direct $\Omega$-factor of $G/U$ and if so, returns  
a direct complement. 
\end{proposition}
\begin{proof}
First compute $C/U=C_{G/U}(V/U)$ via \cite{KL:quo}*{P6},  and test whether $G=\langle C,V\rangle$, for example,  
by computing group orders. If $G\ne \langle C,V\rangle$, then report that $V/U$ is not a direct 
$\Omega$-factor of $G/U$.  Otherwise, compute the center $Z(V/U)$ via \cite{KL:quo}*{P6} and use {\sf $\Omega$-ComplementAbelian}
to compute a $G$-complement $K/U$ to $Z(V/U)$ in $C/U$, or, if none exists, report that $V/U$ is 
not a direct $\Omega$-factor of $G/U$. We prove this this is correct. If $G/U=K/U\times V/U$ is a direct product of $\Omega$-subgroups with $U\leq K\leq G$, then 
$K/U\leq C_{G/U}(V/U)=C/U$ and $K/U$ complements $V/U\cap C/U=Z(V/U)$; the algorithm constructs 
such an $\Omega$-complement.  Conversely, if we find a $\Omega$-complement $K/U$ to $Z(V/U)$ 
in $C/U$, then we have $(K/U)\cap (V/U)=U/U$, and $K/U$ and $V/U$ centralize each other; therefore so long 
as $G/U=\langle K/U,V/U\rangle$, the $\Omega$-subgroup $K/U$ is a direct complement to $V/U$ in $G/U$. We only applied polynomial-time algorithms.
\end{proof}

\subsection{Sylow towers and socles}

Following \cite{rob}*{Section~9.1}, a set of Sylow subgroups, one for each prime dividing the group order, is a \emph{Sylow basis} if any two such  subgroups $U$ and $V$ are permutable, that is, if $UV=VU$;  every solvable group admits a Sylow basis. A group $L$ has an \emph{abelian Sylow tower} if there exists a Sylow basis $\{Y_1,\ldots,Y_\ell\}$ of  abelian groups such that $L=Y_{1}\ltimes\cdots \ltimes Y_{\ell}$.

\begin{proposition}\label{prop:Sylow-tower}
Let $L$ be a solvable group which has an abelian Sylow tower. There is a polynomial-time algorithm that 
computes a Sylow tower $L=Y_{1}\ltimes\cdots \ltimes Y_{\ell}$.
\end{proposition}
\begin{proof}
Compute and factorize $|L|=p_1^{e_1}\cdots p_{\ell}^{e_{\ell}}$. By assumption, $L$ has a normal Sylow subgroup; we run over the prime factors $p_i$ and
compute a Sylow $p_i$-subgroup $P_i$ until  $[P_i,L]$ is contained in $P_i$; if so, set $Y_{\ell}=P_i$. Since all Sylow subgroups are abelian, we use {\sf $\Omega$-ComplementAbelian}  to compute a complement $K\leq L$ to $Y_{\ell}$. By construction, $L=K\ltimes Y_\ell$, and $|K|$ and $|Y_\ell|$ are coprime. Since $K\cong L/Y_\ell$ has an abelian Sylow tower, we can recurse with $K$ and compute a Sylow basis for $K$. We only apply polynomial-time algorithms at most $\sum_{i=1}^\ell i\in O((\log |G|)^2)$ times. 
\end{proof}

We also  need the ability to compute the socle of a solvable group.  Algorithms
for that have been given for permutation groups by Luks \citelist{\cite{luks}\cite{KL:quo}*{P15}} 
and for black-box solvable groups by H\"ofling~\cite{socle}.  H\"ofling's
algorithm reuses the ingredients given above for computing complements, which we will 
later use to construct Frattini subgroups.  So we pause to note the complexity of H\"ofling's
algorithm.

\begin{proposition}\label{prop:socle}
Generators for the socle of a solvable group can be computed in polynomial-time.
\end{proposition}
\begin{proof}  
Let $L$ be a solvable group, treated as an $L$-group under conjugation action. Use \cite{KL:quo}*{P11} to compute a chief series $1=N_0\lhd N_1\lhd \ldots\lhd N_r=L$; in particular, $N_1$ is a minimal normal subgroup of $L$. 
We set $S_1=N_1$, and for each $i>1$ compute a direct $L$-complement $S_i$ to $N_{i-1}$ in $N_i$ (so $S_i\unlhd L$); set $S_i=1$ if this does not exist.    To this end, we proceed as follows: we use the algorithm of Proposition \ref{prop:direct} to find an $L$-subgroup $T\leq N_{i}$ such that $N_{i}=T\times N_{i-1}$; if no such $T$ exists, then we set $S_i=1$. As $T$ is normal in $L$, set $S_i=T$. Once this is done for $i=1,\ldots,r$, return $S_1\times\cdots\times S_r$. The correctness of this algorithm follows from  \cite{socle}*{Proposition~5} where it is shown that  ${\rm soc}(L)=S_1\times\cdots\times S_r$. We only apply algorithms assumed or shown to be polynomial-time.
\end{proof}

\subsection{Computing polycyclic constructive presentations} 
Constructions of polycyclic presentations from solvable permutation groups are done
by various means, sometimes invoking steps (such as \emph{collection}) whose complexities 
are difficult to analyze; see for instance \cite{Seress}*{p.\ 166}. In that approach, 
one first chooses a polycyclic generating sequence $x_1,\dots,x_s$ and then uses the 
constructive membership testing mechanics of permutation groups to sift the relations 
$x_i^{p_i}$ and $x_i^{x_j}$ into words in the $x_k$.  That process leaves the resulting
words in arbitrary order, rather than in collected order, that is, we need 
$x_i^{p_i}=x_{i+1}^{e_{i+1}}\cdots x_s^{e_s}$, but all we can know is that $x_i^{p_i}$
is a word in $x_{i+1},\ldots,x_s$ in no particular order.  Hence, in that
approach, a final step of rewriting must be applied to get the words in normalised 
(\emph{collected}) form; this comes at a cost, see the discussion in \cite{collection}. 
We present an alternative.

\begin{lemma}\label{lem:permpc}
A polycyclic constructive presentation for a solvable group can be computed in polynomial-time.
\end{lemma}
\begin{proof}
Let $L$ be a solvable group. Use \cite{KL:quo}*{P11} to construct a chief series $L=L_0>\ldots >L_s=1$. Since $L$ is solvable, each section 
$L_i/L_{i+1}$ is isomorphic to $C_{p_i}^{f_i}$ for some prime $p_i$ and $f_i\geq 1$. In the following,
set $d(i)=f_0+\ldots+f_{i-1}$ for $i>0$, and denote by $F_{m}$ with $m\in\mathbb{N}$
the free group on $x_1,\ldots,x_{m}$.  We work with a double recursion through $L/L_i$ and within each factor $L_i/L_{i+1}$.

For the inner recursion  we assume $L_i/L_{i+1}\cong C_{p_i}^{f_i}$ 
and want to create a constructive presentation for this group.  Note that 
every chief series of $L_i/L_{i+1}$ is a composition series, so we use  \cite{KL:quo}*{P11} to find generators $g_1,\ldots,g_{f_i}$ of a composition series 
$L_{i0}>L_{i1}>\cdots>L_{if_i}=L_{i+1}$ such that each $L_{ij}=\langle g_{j+1},L_{j+1}\rangle$ and 
$\langle x_j \mid x_j^{p_i}\rangle$ is a presentation for $L_{ij}/L_{i(j+1)}\cong C_{p_i}$.
To make this constructive, use $\psi_j\colon L_{ij}\to F_1$, defined by sending $gL_{i(j+1)}\in L_{ij}/L_{i(j+1)}$ to $x_1^e$ where $g^{-1}g_{j+1}^e\in L_{i(j+1)}$.  Since $e\leq p_1$ is less than the size of the input, $\psi_j$ can be evaluated in polynomial time. This yields a  constructive polycyclic presentation of $L_{ij}/L_{i(j+1)}$. Now suppose by induction we have a constructive polycyclic presentation 
$F_{j}\to L_{i0}/L_{ij}$.  Since we also have a constructive polycyclic presentation of
$F_1\to L_{ij}/L_{i(j+1)}$,  we obtain a constructive presentation $F_{j+1}\to L_{i0}/L_{i(j+1)}$ by Luks' constructive presentation extension lemma \cite{Luks:mat}*{Lemma~4.3}. In that new presentation, every polycyclic relation (for example
$x_k^{p_i}=x_{k+1}^*\cdots x_{j}^*$ or $x_k^{x_{\ell}} = x_{\ell+1}^*\cdots x_j^*$) is 
appended with an element of $\langle x_{j+1}\rangle$, and so the resulting relations are in collected form. Thus, at the end of this inner recursion we have
a polycyclic constructive presentation for the elementary abelian quotients $L_i/L_{i+1}$.

Now consider  the outer recursion. In the base case $i=0$ we apply the above method to
create a constructive polycyclic presentation of $L_0/L_1$.  Now suppose by 
induction we have a polycyclic constructive presentation of $L/L_i$
with maps  $\varphi\colon F_{d(i)}\to L/L_i$ and $\psi\colon L/L_i\to F_{d(i)}$ which can be applied
in polynomial time. As in the base case, we construct a polycyclic constructive 
presentation with maps $\varphi'\colon F_{f_i}\to L_i/L_{i+1}$ and 
$\psi'\colon L_i/L_{i+1}\to  F_{f_i}$.  
Luks' extension lemma now makes a constructive presentation for $L/L_{i+1}$ with maps $\varphi^*\colon F_{d(i+1)}\to L/L_{i+1}$ and  
$\psi\colon L/L_{i+1}\to F_{d(i+1)}$. In this process, relations of $L/L_i$ of the form $x_j^{p}=x_1^*\cdots x_{d(i)}^*$ and
$x_j^{x_k} = x_{k+1}^*\cdots x_{d(i)}^*$ are appended with normalised words in $L_i/L_{i+1}$, 
so these continue to be in collected form. We also add the polycyclic relations 
for $L_i/L_{i+1}$, so the extended constructive presentation is  polycyclic.
\end{proof}  



\section{Isomorphism testing of cube-free groups: solvable Frattini-free groups}\label{secsolvFF}\label{sec:solvFF}
We now deal with Step (2) of our algorithm as described in Section  \ref{sec:summary}. 
Using the notation of  Section~\ref{secStructure}, throughout the following $L$ and $\tilde L$ 
are  finite solvable groups of cube-free order, and we consider their Frattini-free quotients 
$L_{\Phi}=L/\Phi(L)$ and $\tilde L_{\Phi}=\tilde L/\Phi(\tilde L)$. Recall that 
$L_{\Phi}=K\ltimes {\rm soc}(L_\Phi)$ with ${\rm soc}(L_\Phi)=B\times C$ where $|B|=b$
and $|C|=c^2$ with $b$ and $c$ square-free; analogously for $\tilde L_\Phi$. 
In the remainder of this section we describe how to construct an isomorphism 
$L_\Phi\to \tilde L_\Phi$; our construction fails if and only if the two groups are not isomorphic. 
  
\begin{proposition}\label{prop:frat-free-decomp}
There is a polynomial-time algorithm given a solvable Frattini-free group $L_\Phi$ 
of cube-free order, returns generators for the decomposition into subgroups
$(K,B,C)$ described above, along with isomorphisms $B\to \prod_{i=1}^s \mathbb{Z}/p_i$ 
and $C\to \prod_{j=s+1}^m (\mathbb{Z}/p_j)^2$, and
a representation \[K\to \Aut(B\times C)\to \prod\nolimits_{i=1}^s \GL_1(p_i)\times \prod\nolimits_{j=s+1}^m\GL_2(p_j)\] 
induced by conjugation of $K$ on $B\times C$.
\end{proposition}
\begin{proof} Use the algorithms of Propositions \ref{prop:socle} \& \ref{prop:complement} 
to compute generators for ${\rm soc}(L_\Phi)$ and for a complement $K$ to ${\rm soc}(L_\Phi)$ 
in $L_\Phi$.  Then use the algorithm of Proposition~\ref{prop:Sylow-tower} to decompose 
${\rm soc}(L)$ as a direct product of its Sylow subgroups. Using the decomposition 
series of each Sylow subgroup, we obtain the decomposition ${\rm soc}(L_\Phi)=B\times C$ along with 
primary decompositions of $B=\prod_{i=1}^s Y_i$ and $C=\prod_{j=s+1}^m Y_j$.  We can further produce isomorphisms  
$\beta_i:Y_i\to \mathbb{Z}/p_i$ and $\kappa_j:Y_j\to (\mathbb{Z}/p_j)^2$, for example, by using our results from {\cite{PART1}*{Section 3}}, based on Karagiorgos \& Poulakis \cite{KP}.  Given 
standard representations for $\Aut(\mathbb{Z}/p_i)\cong (\mathbb{Z}/p_i)^{\times}$ and
$\Aut((\mathbb{Z}/p_j)^2)=\mathrm{GL}_2(p_j)$, compose with $\beta_i$ and $\kappa_j$
respectively to produce an isomorphism
\begin{align*}
	\tau\colon\mathrm{Aut}(B\times C)\to \prod\nolimits_{i=1}^s \GL_1(p_i)\times \prod\nolimits_{j=s+1}^m \mathrm{GL}_2(p_j).
\end{align*}
Finally, define $\pi:K\to\Aut(B\times C)$ by $(bc)(k)\pi= b^kc^k$, so $\pi\tau$ is the required map from $K$. 

The correctness of this algorithm is apparent. 
The claim on the timing of the first portion follows since we only invoked $O(\log |L_\Phi|)$ many 
polynomial-time algorithms.  We can apply the algorithms of {\cite{PART1}*{Section 3}} 
to construct an isomorphism in polynomial time since  $|Y_i|=p_i$ and $|Y_j|=p_j^2$, and
both $p_j$ and $p_j$ are bounded by the size of the permutation domain $\Omega$ of $L$.
So the complexity of the results used from {\cite{PART1}} is sufficient.  Our  assumption 
is that all groups here are permutation groups: in the case of the groups $\prod_{j=s+1}^m \GL_2(p_j)$, we can treat  the matrices as permutations of pairs $\bigcup_{j=s+1}^m \{(a,b) | a,b\in \mathbb{Z}/p_j\}$;  
this domain has size $O(p_{s+1}+\cdots+p_m)\subset O(|\Omega|\log |L|)$, so is polynomial
in the input size.
\end{proof}

To simplify the exposition, we make the following convention and identify 
\begin{eqnarray*}
	B=\tilde B=\prod\nolimits_{i=1}^s \mathbb{Z}/{p_i}
		\quad\text{and}\quad
		 C=\tilde C=\prod\nolimits_{i=s+1}^m (\mathbb{Z}/{p_i})^2.
\end{eqnarray*}
Recall from Section~\ref{secStructure} that  the conjugation action of $K$ on $B\times C$ is 
faithful.  Hence, we also treat $K$ and $\tilde K$ as subgroups of
\[\Aut(B\times C)=\prod\nolimits_{i=1}^s {\rm GL}_1(p_i)\times \prod\nolimits_{i=s+1}^m {\rm GL}_2(p_i).\]
For $j=1,\ldots,m$ denote by $K_i$ and $\tilde K_i$ the projections of $K$ and 
$\tilde K$, respectively, into the $j$-th factor of $\Aut(B\times C)$; thus $K_j$ 
and $\tilde K_j$ describe the conjugation action of $K$ and $\tilde K$, respectively, 
on the Sylow $p_j$-subgroup $Y_{j}\leq B\times C$. 

Gasch\"utz has shown that  
$L_\Phi\cong \tilde L_\Phi$ if and only if $K$ and $\tilde K$ are conjugate in 
$\Aut(B\times C)$, see \cite{cf}*{Lemma~9}; hence, the isomorphism problem reduces 
to finding an element $\alpha\in\Aut(B\times C)$ with $\alpha^{-1} K\alpha=\tilde K$. 
Once such an $\alpha$ is found, the isomorphism $\psi_\Phi$ can be defined 
as follows:  writing the elements of $L_\Phi=K\ltimes (B\times C)$ and 
$\tilde L_\Phi=\tilde K\ltimes (B\times C)$ as $(k,b,c)$ and $(\tilde k,b,c)$, respectively, we set
\begin{align}
	\psi_\Phi\colon L_\Phi\to \tilde L_\Phi,\quad (k,b,c)\mapsto (\alpha^{-1}k\alpha,b^\alpha ,c^\alpha).
\end{align}
Our  construction of $\alpha$ depends very much on the dimension $2$ case; in particular, we use 
a classification of J.\ Gierster (1881) of the subgroups of $\mathrm{GL}_2(p)$, extracted from  \cite{gl2}*{Theorems~5.1--5.3}.
\begin{lemma}\label{lemgl2}
Let $p$ be an odd prime and let $K\leq \GL_2(p)$ be a solvable cube-free $p'$-subgroup.
\begin{ithm} 
\item If $K$ is reducible, then $K$ is conjugate to a subgroup of diagonal matrices.
\item If $K$ is irreducible and abelian, then $K$ is conjugate to $\langle s^{(p^2-1)/r}\rangle$ for some $r\mid p^2-1$, where $s$ is a generator of a Singer cycle in $\GL_2(p)$, that is, $\langle s\rangle \cong C_{p^2-1}$.
\item If $K$ is irreducible and non-abelian, then there are three possibilities. First, $K$ might be  conjugate to $G_2\ltimes G_{2'}$ where $G_{2'}$ is an odd order diagonal (but non-scalar) subgroup  and $G_2$ is one of \[\langle\left(\begin{smallmatrix}0&1\\1&0\end{smallmatrix}\right)\rangle,\quad 
  \langle\left(\begin{smallmatrix}0&z\\z&0\end{smallmatrix}\right)\rangle,\quad 
   \langle\left(\begin{smallmatrix}0&-1\\1&0\end{smallmatrix}\right)\rangle,\quad
     \langle\left(\begin{smallmatrix}0&1\\1&0\end{smallmatrix}\right),\left(\begin{smallmatrix}-1&0\\0&-1\end{smallmatrix}\right)\rangle,
\] 
with $z\in\mathbb{Z}/p$ of order 4 (if it exists). Second, $K$ might be conjugate to $\langle S,t\rangle$ where $S$ is a subgroup of a Singer cycle $\langle s\rangle$ and $t$ is an involution such that $N_{\GL_2(p)}(\langle s\rangle)=\langle s,t\rangle$. Third, $K$ might be conjugate to $\langle S,ts^{2l}\rangle$ where $S\leq \langle s\rangle$ has even order and $p-1=4l$ with $l$ odd. 
\end{ithm}
In particular, $N_{\GL_2(p)}(K)/C_{\GL_2(p)}(K)$ is solvable.
\end{lemma}  

We further need an algorithm of Luks \& Miyazaki's \cite{LM:Normalizer} that
demonstrates how to decide conjugacy of subgroups in solvable permutation groups 
in time polynomial in the input size.

\begin{theorem}\label{thm:Conj-GL2s}
Let $G$, $K$, and $\tilde{K}$ be groups with
\begin{align*}
	K,\tilde K & \leq G=\langle S\rangle =\prod\nolimits_{i=1}^n \mathrm{GL}_2(p_i),
\end{align*}
where $K$ and $\tilde K$ are solvable groups of equal cube-free order coprime to $p_1\cdots p_n$. 
One can decide in polynomial time whether $K$ is conjugate to $\tilde K$ and produce a conjugating element, if it exists.
\end{theorem} 
\begin{proof}
As above, let $K_i$ and $\tilde K_i$ be the projections of $K$ and $\tilde{K}$, respectively, to the factor
$\mathrm{GL}_2(p_i)$.   For each $i$, based on the classification given in Lemma \ref{lemgl2}, we apply basic linear algebra methods to solve for  $\alpha_i\in\mathrm{GL}_2(p_i)$
such that $K_i^{\alpha_i}=\tilde K_i$; we also construct $N_i=N_{\GL_2(p_i)}(\tilde{K}_i)$ based on Lemma~\ref{lemgl2}. If we cannot find a particular $\alpha_i$, then $K$ and $\tilde K$ are not conjugate and we return that. Once all the $\alpha_i$ have been computed, we replace $K$ by $K=K^{\alpha_1\cdots \alpha_n}$, so that we can assume that $K_i=\tilde K_i$ for all $i$. Note that  $K$ and $\tilde K$ are conjugate if and only if they are conjugate in $N=\prod_{i=1}^n N_i$, which is
solvable by Lemma~\ref{lemgl2}. Now we apply the algorithm of \cite{LM:Normalizer}*{Theorem~1.3(ii)} 
to solve for $\beta\in N$ such that $K^{\beta}=\tilde{K}$, and return $\alpha_1\cdots\alpha_n\beta$. If we cannot find such a $\beta$, then $K$ and $\tilde{K}$ are not conjugate, and we return false. Lastly, we comment on the timing. Note that we can also locate appropriate $\alpha_i$ by a polynomial-time brute-force search in $\mathrm{GL}_2(p_i)$: the latter has order at most $p_i^4\leq d^4$, where $d$ is the size of the permutation domain of $G$. We make a total of $n\leq \log |G|$ such searches, followed by the polynomial-time algorithm of \cite{LM:Normalizer}. The claim follows.
\end{proof}


\section{Isomorphism testing of cube-free groups: solvable groups} \label{sec:solv}
Throughout this section $L$ and $\tilde L$ are  finite solvable groups of cube-free 
order, given as permutation groups. To decide isomorphism, we first want to use the 
algorithm of Section~\ref{secsolvFF} to determine whether the Frattini quotients 
$L_\Phi$ and $\tilde L_\Phi$ are isomorphic. For this we need the Frattini subgroups.

\subsection{Frattini subgroups}\label{remPhi} 
Since we assume  permutation groups as input, we need a 
polynomial-time algorithm to compute Frattini subgroups of solvable permutation 
groups of cube-free order. A candidate algorithm has been provided by Eick  \cite{Phi}*{Section~2.4} for groups 
given by a polycyclic (pc) presentation. To adapt to a permutation setting we 
have two choices: replace every step of that algorithm with polynomial-time
variants for permutation groups, or apply the algorithm 
{\em in-situ} by appealing to a two-way isomorphism between our original permutation 
group and a constructive pc-presentation as afforded to us by Lemma \ref{lem:permpc}.
Note that for the efficiency of the inverse isomorphism, elements in a pc-group are 
straight-line programs (SLPs) in the generators, so evaluation is determined on the generators 
and computed in polynomial time.  Thus, whenever we take products in the pc-group, we 
actually carry out permutation multiplications and sift these into the polycyclic 
generators by applying the isomorphism back to the pc-group.  This avoids the 
potential exponential complexity of collection in pc-groups, see the 
discussion in~\cite{collection}.  That the algorithm in \cite{Phi}*{Section~4.2} uses a polynomial number of pc-group
operations follows by considering its major steps.  It relies on constructing 
complements of abelian subgroups (shown in Proposition~\ref{prop:complement} to be in polynomial time), 
and it applies also module decompositions (which can be done in polynomial time see 
\cite{LM:Normalizer}*{Theorem~3.7 \& Section~3.5}), and finally computing cores
\cite{KL:quo}*{P5}.  Therefore Eick's algorithm is in fact a polynomial-time algorithm for groups of permutations, 
and we cite it as such in what follows.

Once $\Phi(L)$ and $\Phi(\tilde L)$ have been constructed, we can compute the 
quotients  $L_\Phi$ and $\tilde L_\Phi$, see \cite{KL:quo}, and use the algorithms 
of Section \ref{secsolvFF} to test isomorphism. If we have determined
that $L_\Phi\not\cong \tilde L_\Phi$, then we can report that $L\not\cong \tilde L$. 
Thus, in the following we assume we found an isomorphism 
$\varphi\colon L_\Phi\to\tilde L_\Phi$, so we also know that $L\cong \tilde L$ by Section \ref{secStructure}. In the next sections we  describe how to 
construct an isomorphism $\hat{\varphi}\colon L\to \tilde L$ such that $\hat{\varphi}$ 
factors through $\varphi$ in the sense that $\Phi(\tilde{L})(g\hat{\varphi}) = (\Phi(L)g)\varphi $ 
for all $g\in L$.  This  condition is what allows us to not only solve for some 
isomorphism between $L$ and $\tilde{L}$, but to also lift generators for the automorphism
group of $L$ and thus prescribe (generators for) the entire coset of isomorphisms 
$L\to \tilde{L}$.  Our approach to computing $\hat{\phi}$ is to work with each prime 
divisor of $|\Phi(L)|$.  We begin with a key observation about these primes and recall 
the Frattini extension structure of  groups of cube-free order.

\subsection{Frattini extension structure}\label{secFES}

As above, write $A_1\Yup\ldots \Yup A_s$ for 
any subdirect product of groups $A_1,\ldots,A_s$. For a group $Y$ and prime $p$ dividing $|Y|$ let $Y_p$ be a Sylow $p$-subgroup of $Y$. It follows from \cite[9.2]{rob} that every finite solvable group has a Sylow basis, and it follows from  \cite{curran,cf2} that every solvable cube-free group $Y$ has one of the following abelian Sylow towers
\begin{eqnarray*}
Y=\begin{cases} 
Y_{r_1}\ltimes Y_{r_2}\ltimes \ldots \ltimes Y_{r_\ell}            &{\text{if }|Y|\text{ odd}}\\
Y_2\ltimes Y_{r_1}\ltimes Y_{r_2}\ltimes \ldots \ltimes Y_{r_\ell} &{\text{if }|Y|\text{ even, $Y_2\not\!\!\unlhd Y$}}\\
Y_{r_1}\ltimes Y_{r_2}\ltimes \ldots \ltimes Y_{r_\ell}\ltimes Y_2\; \text{(with $Y_2=C_2^2$)}&{\text{if }|Y|\text{ even, $Y_2\unlhd Y$}}
\end{cases}
\end{eqnarray*}
where $r_1<\ldots<r_\ell$ are the odd prime divisors of $|Y|$ and $\{(Y_2),Y_{r_1},\ldots,Y_{r_\ell}\}$ forms a Sylow basis of $Y$. Proposition \ref{prop:Sylow-tower} provides an algorithm to construct such a Sylow tower.

\begin{lemma}\label{lemFEF} 
Let $L$ be  Frattini-free and solvable, and let $Y^\ast$ be a cube-free Frattini extension of $Y$, that is, $Y^\ast/\Phi(Y^\ast)\cong Y$. If $p\nmid |\Phi(Y^\ast)|$, then $Y^\ast_p\cong Y_p$; otherwise $Y_p\cong C_p$ and $Y^\ast_p\cong C_{p^2}$.
\end{lemma}
\begin{proof}  
 Recall  that every prime dividing $|\Phi(Y^\ast)|$ must divide $|Y|$, thus $\Phi(Y^\ast)$ is square-free and the Sylow tower of $Y^\ast$ looks similar to that of $Y$, where $Y^\ast_p\cong Y_p$ if $p\nmid |\Phi(Y^\ast)|$, and $Y_p\cong C_p$ and $Y^\ast_p$ abelian of order $p^2$ otherwise. We prove that $Y^\ast_p\cong C_{p^2}$.   We use the previous notation and consider $M=\Phi(Y^\ast)=C_{p_1}\times\ldots\times C_{p_m}$ as a $Y$-module. It is shown in \cite[Theorem 12]{cf} that $Y^\ast$ is a subdirect product of Frattini extensions of $Y$ by $C_{p_i}$. Thus, to prove the lemma, it suffices to consider $M=C_p$  for some prime $p$. First, suppose that $p=r_i$ is odd. In this case, $Y_p\cong C_p$ and $Y^\ast_p$ is abelian of order $p^2$. Suppose, for a contradiction, that $Y^\ast_p\cong C_p^2$.   It follows from \cite[Lemma~5 \& Theorem 14]{cf} that $Y^\ast$ is a non-split extension of $Y$ by $M$ such that $N_Y(Y_p)$ acts on $M$ as on $Y_p$. This implies the following: considering  $Y^\ast_p=Y^\ast_{r_i}=(\mathbb{Z}/p)^{2}$ as an $\mathbb{Z}/p$-space, there is a basis $\{m,y\}$ such that $M=\langle m\rangle$ and every $g\in  (Y_2^\ast\ltimes) Y^\ast_{r_1}\ltimes \ldots\ltimes Y^\ast_{r_{i-1}}\leq N_{Y^\ast}(Y^\ast_p)$ acts on that space as a matrix $\tilde g=\left(\begin{smallmatrix} \alpha&\beta\\0&\alpha\end{smallmatrix}\right)$ for some $\alpha\in (\mathbb{Z}/p)^\times$ and $\beta\in \mathbb{Z}/p$. Since $|Y^\ast|$ is cube-free, $g$ has order coprime to $p$, and hence   $\beta=0$, that is, $g$ acts diagonally on $Y^\ast_p$. 
    Moreover, $W=Y^\ast_{r_{i+1}}\ltimes\ldots\ltimes Y_{r_\ell}^\ast(\ltimes Y_2^\ast)$ centralizes $Y^\ast_p$ modulo $W$. In conclusion, no nontrivial element in $Y^\ast_p$ is a non-generator of $Y^\ast$, contradicting $\Phi(Y^\ast)\leq Y^\ast_p$, see \cite[Proposition 2.44]{handbook}. This contradiction proves $Y^\ast\cong C_{p^2}$. Lastly, suppose $M=C_2$; in this case  $Y_2\cong C_2$ and $Y^\ast=Y^\ast_2\ltimes Y^\ast_{r_1}\ltimes\ldots \ltimes Y^\ast_{r_\ell}$. If $Y^\ast_2\cong C_2^2$, then the same argument  shows that no nontrivial element in $Y^\ast_2$ is a non-generator of $Y^\ast$, contradicting $\Phi(Y^\ast)\leq Y^\ast_2$. Thus, $Y^\ast_2\cong C_4$. 
\end{proof}

\subsection{Constructing the isomorphism}

Recall that $L\cong \tilde L$ if and only if the isomorphism $\psi_\Phi$ in Step (2) exists. Suppose $\psi_\Phi$ has been constructed as described in Section \ref{secsolvFF}, that is, we know that $L\cong \tilde L$. As explained in the proof of Lemma \ref{lemFEF}, the groups $L$ and $\tilde L$ are iterated Frattini extensions of $L_\Phi$ and $\tilde L_\Phi$, respectively, by cyclic groups of prime order; cf.\ \cite[Definition 4]{cf}. Starting with $\psi_\Phi$, we iteratively construct isomorphisms of these Frattini extensions until eventually we obtain an isomorphism $L\to \tilde L$. Thus, we consider the following situation: let $Y$ and $\tilde Y$ be two solvable cube-free groups and let $Y^\ast$ and $\tilde Y^\ast$ be cube-free Frattini extensions of $Y$ and $\tilde Y$, respectively, by $M=C_p$. We assume that we have an isomorphism $\varphi\colon Y\to \tilde Y$; we know that $Y^\ast\cong \tilde Y^\ast$, and we aim to construct an isomorphism $Y^\ast\to\tilde Y^\ast$. The following preliminary lemma will be handy.

\begin{lemma}\label{lemOrd}
Let $G$ be a group and $P,Q\leq G$ such that $P$ is a cube-free $p$-group and $Q=\langle w\rangle$ is cyclic of order $q^2$, for distinct primes $p$ and $q$. Suppose $PQ=QP$ and $A=\langle w^q\rangle$ is normal in~$PQ$.
\begin{ithm}
\item We have $PQ=P\ltimes Q$ or $PQ=Q\ltimes P$. 
\item If $PQ=Q\ltimes P$, then $A$ acts trivially on $P$. 
\item If $PQ=P\ltimes Q$, then the action of $P$ on $Q$ is uniquely determined by its action on $Q/A$. 
\end{ithm}
\end{lemma}   
\begin{proof}
Since $PQ$ is cube-free, part a) follows from the structure results mentioned in Section~\ref{secFES}. For part b), note that $Q$ and $Q/A$ both act on $P$; this forces that $A$ acts trivially on $P$. Now consider part c). Recall that $\Aut(Q)$ is cyclic of order $q(q-1)$, generated by $\beta\colon Q\to Q$, $w\mapsto w^k$, where $k$ is some primitive root modulo $q^2$. Since $PQ$ is cube-free, the element $g\in P$ acts on $Q$ via an automorphism $\alpha\in\Aut(Q)$ of order coprime $q$.  Thus, $\alpha$  lies in the subgroup $T\leq \Aut(Q)$ of order $q-1$, and there is a unique $e\in\{1,\ldots,q-1\}$ such that $\alpha=(\beta^q)^e$. Now $(wA)\alpha= (wA)^i$ with $i\in\{0,\ldots,q-1\}$ yields $i=k^{eq}\bmod q$.  Since $k^q$ is a primitive root modulo $q$, it follows that for any given $i\in\{1,\ldots,q-1\}$ there is a unique $e\in\{1,\ldots,q-1\}$ such that $i\equiv k^{eq}\bmod p$, hence for a given $i\in\{0,\ldots,q-1\}$ there is a unique automorphism $\alpha\in\Aut(Q)$ with $(wA)\alpha=(wA)^i$.
\end{proof}

\begin{proposition}\label{prop:cyclicFE}
Let $Y$ and $\tilde Y$ be two solvable cube-free groups and let $Y^\ast$ and $\tilde Y^\ast$ be cube-free Frattini extensions of $Y$ and $\tilde Y$, respectively, by a group isomorphic to $C_p$.  Algorithm \ref{algo:lift-cyclic} is  a polynomial-time algorithm which, given an isomorphism $\varphi\colon Y\to \tilde Y$, returns an isomorphism $\hat\varphi\colon Y^\ast\to\tilde Y^\ast$.
\end{proposition} 
\begin{proof} We compute the Frattini subgroups of $Y^\ast$ and $\tilde Y^\ast$, and the Sylow $p$-subgroups $A\leq \Phi(Y^\ast)$ and $\tilde A\leq\Phi(\tilde Y^\ast)$, respectively, see Section \ref{remPhi}. By assumption, $A\cong \tilde A\cong C_p$, and we can assume that $Y=Y^\ast/A$ and $\tilde Y=\tilde Y^\ast/\tilde A$. As explained above, the existence of $\varphi\colon Y\to \tilde Y$ implies that $Y^\ast$ and $\tilde Y^\ast$ are isomorphic. Use the algorithm of Proposition \ref{prop:Sylow-tower} to construct  a Sylow tower  $Y^\ast=Y_1^\ast\ltimes \ldots \ltimes Y_n^\ast$; for each $j$ let $p_j$ be a prime such that $Y_j^\ast$ is a Sylow $p_j$-subgroup. Let $p=p_i$, and recall from Lemma~\ref{lemFEF} that $Y_i^\ast$  is  cyclic; find a generator $Y_i^\ast=\langle a\rangle$ and note that $A=\langle a^p\rangle\leq Y_i^\ast$. For every $j$ define $Q_j=\prod\nolimits_{k\ne j} Y_k^\ast$; this is a Hall $p_j'$-subgroup of $Y^\ast$. Such a set of Hall $r'$-subgroups (one for each prime divisor $r$ of the group order) is a called a {\it Sylow system} in \cite[Section 9.2]{rob}); in particular,  we can recover each $Y_j^\ast$ as  $Y_j^\ast=\bigcap_{k\ne j}Q_k$. 
   
 Since $Y_1^\ast,\ldots,Y_n^\ast$ form a Sylow tower of $Y^\ast$, every $x\in Y^\ast$ has a unique factorization $x=h a^{e}$ where $h\in H=Q_i$ and $a^e\in Y_i^\ast$ with $0\leq e\leq p^2-1$; we will use this decomposition later when we define an isomorphism $\hat\varphi\colon Y^\ast\to \tilde Y^\ast$. We will construct $\hat\varphi$  via a Sylow basis of $\tilde Y^\ast$ which is {\em compatible} with the above Sylow basis of $Y^\ast$; we explain below what this means.
 
Let $\Gamma\colon Y^\ast\to Y^\ast/A=Y$ be the natural projection, so that $\{Q_1\Gamma\varphi,\ldots, Q_n\Gamma\varphi\}$ forms  a Sylow system of $\tilde Y^\ast/\tilde A$. For each $j$ we define  $\tilde Q_j\leq \tilde Y^\ast$ to be the full preimage of $Q_j\Gamma\varphi$ under the natural projection $\tilde\Gamma\colon \tilde Y^\ast\to \tilde Y^\ast/\tilde A=\tilde Y$. Clearly, if $j\ne i$, then $\tilde Q_j$ is a Hall $p_j'$-subgroup of $\tilde Y^\ast$. Moreover, $\tilde Q_i= \tilde H \ltimes \tilde A$ where $\tilde H$ is some Hall $p'$-subgroup of $\tilde Q_i$ and of $\tilde Y^\ast$; we compute $\tilde H$ in $\tilde Q_i$ by first computing $\tilde A\leq \tilde Q_i$ as a Sylow $p$-subgroup and then $\tilde H$ as a complement to $\tilde A$ in $\tilde Q_i$. We define $\tilde Y_i^\ast=\bigcap\nolimits_{k\ne i} \tilde Q_k$ and \[\tilde Y_j^\ast = \tilde H\cap \bigcap\nolimits_{k\ne j,i} \tilde Q_k\quad\text{for each $j\ne i$}.\]It follows from \cite[9.2.1]{rob} that $\{\tilde Y_1^\ast,\ldots,\tilde Y_n^\ast\}$ is a set of pairwise permutable Sylow subgroups with $\tilde Y_j^\ast\tilde \Gamma=Y_j^\ast\Gamma\varphi$ for all $j$. In particular, we can apply Lemma \ref{lemOrd} and it follows from our construction that for all $u\ne v$ we have  $\tilde Y_u^\ast\tilde Y_v^\ast=\tilde Y_u^\ast\ltimes \tilde Y_v^\ast$ if and only if $Y_u^\ast Y_v^\ast=Y_u^\ast\ltimes Y_v^\ast$, and $\tilde Y_u^\ast\tilde Y_v^\ast=\tilde Y_v^\ast\ltimes \tilde Y_u^\ast$  if and only if $Y_u^\ast Y_v^\ast=Y_v^\ast\ltimes Y_u^\ast$.  We say that these two Sylow bases are \emph{compatible}.

Let $\pi$ and $\tilde\pi$ be the restriction of $\Gamma$ and $\tilde \Gamma$ to $H$ and $\tilde H$, respectively; note that  $\pi\colon H\to HA/A$ and $\tilde\pi \colon \tilde H\to \tilde H\tilde A/\tilde A$ are isomorphisms, and we define an isomorphism $H\to \tilde H$ via
\begin{align*}
	H=H/(H\cap A)\overset{\pi}{\longrightarrow} 
		HA/A\overset{\varphi}{\longrightarrow}  
		\tilde H\tilde A/\tilde A
		\overset{\tilde\pi^{-1}}{\longrightarrow} \tilde H/\tilde H\cap \tilde A=\tilde H.
\end{align*}
Note that in defining $\tilde \pi\colon h\mapsto \tilde A h$, we identify generators of $\tilde H$ with generators of $\tilde H\tilde A/\tilde A$;  as elements of $\tilde H\tilde A/\tilde A$ are presumed throughout to be words (or SLPs) in the generators, we can compute preimages of $\tilde \pi$. This affords us an implementation of $\tilde\pi^{-1}$. 
 
Recall that $Y_i^\ast=\langle a\rangle$, and choose a generator $\tilde a\in \tilde Y_i^\ast$ such that \[a\Gamma\varphi=\tilde a\tilde\Gamma.\] We can now construct an isomorphism $\hat\varphi\colon Y^\ast\to\tilde Y^\ast$. As mentioned above, every $x\in Y^\ast$ has a unique factorization $x=h a^{e}$ where $h\in H$ and $0\leq e\leq p^2-1$. This shows that\[\hat \varphi\colon Y^\ast\to \tilde Y^\ast,\quad ha^e\mapsto h\pi\varphi\tilde \pi^{-1}\cdot \tilde a^e,\]is well-defined; clearly, $\hat \varphi$ is a bijection, so it remains to show that it is a homomorphism. We use below the important property of $\hat\varphi$ that it maps $Y^\ast_j$ to $\tilde Y^\ast_j$ for each $j$: this follows from the fact that the Hall subgroups $Q_1,\ldots,Q_n$ defining the Sylow basis $Y^\ast_1,\ldots,Y_n^\ast$ are mapped under $\hat\varphi$ to the Hall subgroups $\tilde Q_1,\ldots,\tilde Q_{i-1},\tilde H,\tilde Q_{i+1},\ldots,\tilde Q_n$ defining the Sylow basis $\tilde Y^\ast_1,\ldots,\tilde Y^\ast_n$.  

Let $x,y\in Y^\ast$ and write $x=ha^e$ and $y=ka^f$ with $h,k\in H$ and $e,f\in\{0,\ldots,p^2-1\}$. Write $(a^e)^k=ma^u$ with $m\in H$ and $u\in\{0,\ldots,p^2-1\}$, so that $xy=hk(a^e)^ka^f=(hkm)a^{u+f}$. This shows that \[(xy)\hat\varphi = x\hat\varphi\cdot y\hat\varphi \iff (\tilde a^e)^{k\pi\varphi\tilde \pi^{-1}}=m\pi\varphi\tilde\pi^{-1}\cdot \tilde a^u,\]and it remains to prove the following: for all $k\in H$ and $e\in\{0,\ldots,p^2-1\}$, if  $(a^e)^k=ma^u$ with $m\in H$, then  $(\tilde a^e)^{k\pi\varphi\tilde\pi^{-1}}=m\pi\varphi\tilde\pi^{-1}\cdot \tilde a^u$.  Recall that every $k\in H$ can be written as a product of elements in the chosen Sylow tower of $Y^\ast$, say $k=h_1\ldots h_l$ where $h_u$ and $h_v$ lie in different Sylow subgroups for $u\ne v$. We prove the claim by induction on $l$.

First, suppose $l=1$, that is, $k$ lies in a Sylow $p_j$-subgroup $Y_j^\ast\leq H$ for some $j\ne i$. It follows from Lemma \ref{lemOrd} that $Y_i^\ast Y_j^\ast=Y_j^\ast Y_i^\ast$ is a $\{p,p_j\}$-group, and there are two cases to  consider.

\begin{items}
\item[(i)] If $Y_j^\ast$ normalizes $Y_i^\ast$, then $\tilde Y_j^\ast$ normalizes $\tilde Y_i^\ast$.  We can write $(a^e)^k=a^{i}$ for a uniquely determined  $i\in\{0,\ldots,p^2-1\}$, which  yields \[(Aa)^{k\pi}=(Aa^{i\bmod p})\quad\text{and}\quad (\tilde A\tilde a)^{k\pi\varphi}=(\tilde A\tilde a^{i\bmod p}).\] Since $k$ acts on $\langle Aa\rangle$ the same way as $k\pi\varphi$ acts on $\langle \tilde A\tilde a\rangle$, it follows from Lemma \ref{lemOrd} that $k$ acts on $A$ the same way as $k\pi\varphi\tilde \pi^{-1}$ acts on $\tilde A$. Thus, if $(a^e)^k=a^i$, then $(\tilde a^e)^{k\pi\varphi\tilde\pi^{-1}}=\tilde a^i$, as claimed. 

\item[(ii)] If $Y_i^\ast$ normalizes $Y_j^\ast$, then  $\tilde Y_i^\ast$ normalizes $\tilde Y_j^\ast$. Moreover, $A=\langle a^p\rangle\leq Y_i^\ast$ and $\tilde A=\langle \tilde a^p\rangle\leq \tilde Y_i^\ast$ act trivially on $Y_j^\ast$ and on $\tilde Y_j^\ast$, respectively,  and \begin{align*}(a^e)^k&=[k,a^{-e}]a^e=[k,a^{-e\bmod p}]a^e &&\text{with}\quad [k,a^{-e\bmod p}]\in Y_j^\ast\leq H\\
(\tilde a^e)^{k\pi\varphi\tilde \pi^{-1}}&=[k\pi\varphi\tilde \pi^{-1},\tilde a^{-e\bmod p}]\tilde a^e &&\text{with}\quad [k\pi\varphi\tilde \pi^{-1},\tilde a^{-e\bmod p }]\in \tilde Y_j^\ast\leq \tilde H.
\end{align*}
Thus, it remains to show that $[k,a^{-e\bmod p}]\pi\varphi\tilde\pi^{-1}= [k\pi\varphi\tilde \pi^{-1},\tilde a^{-e\bmod p}]$. Note that  \[[k,a^{-e\bmod p}]\pi\varphi\tilde\pi^{-1}=[k\pi\varphi,\tilde A\tilde a^{-e\bmod p}]\tilde\pi^{-1},\]
and  $[k\pi\varphi\tilde\pi^{-1},\tilde a^{-e\bmod p}]\in \tilde H$ is a preimage of $[k\pi\varphi,\tilde A\tilde a^{-e\bmod p}]\in \tilde H\tilde A/\tilde A$ under the isomorphism $\tilde \pi\colon \tilde H\to \tilde H\tilde A/\tilde A$; recall that $\tilde \pi$ is the restriction of $\tilde \Gamma\colon \tilde Y^\ast\to \tilde Y$, and $\tilde \Gamma$ maps $k\pi\varphi\tilde\pi^{-1}$ and $\tilde a$ to $k\pi\varphi$ and $\tilde A\tilde a$, respectively. Thus, $(\tilde a^e)^{k\pi\varphi\tilde \pi^{-1}}= [k,a^{-e\bmod p}]\pi\varphi\tilde\pi^{-1}\cdot \tilde a^e$, as claimed.  
\end{items}
 Second, consider the induction step $l\geq 2$ and write $k=st$ such that the induction hypothesis holds for $s$ and $t$, that is, if $(a^e)^s=m_s a^{u_s}$ with $m_s\in H$, then $(\tilde a^e)^{s\pi\varphi\tilde \pi^{-1}}= m_s\pi\varphi\tilde\pi^{-1}\cdot \tilde a^{u_s}$, and that if $(a^{u_s})^t=m_t a^{u_t}$ with $m_t\in H$, then  $(\tilde a^{u_s})^{t\pi\varphi\tilde\pi^{-1}}= m_t\pi\varphi\tilde\pi^{-1}\cdot \tilde a^{u_t}$. This  yields $(a^e)^k=m_s^tm_t a^{u_t}$ with $m_s^t m_t\in H$, and therefore \[(\tilde a^e)^{k\pi\varphi\tilde \pi^{-1}}=(m_s\pi\varphi\tilde\pi^{-1})^{t\pi\varphi\tilde\pi^{-1}}\cdot m_t\pi\varphi\tilde\pi^{-1}\cdot \tilde a^{u_t}=(m_s^tm_t)\pi\varphi\tilde\pi^{-1}\cdot \tilde a^{u_t},\] as claimed. This completes the proof that $\hat\varphi$ is an isomorphism between $Y^\ast$ and $\tilde Y^\ast$.  The construction of $\hat\varphi$  only employs a finite list of polynomial-time algorithms.
\end{proof}

As explained in the beginning of this section, if the order of the cube-free group $L$ has $k$ distinct prime divisors, then the algorithm in Proposition \ref{prop:cyclicFE} has to be iterated at most $k$ times to establish an isomorphism from $L$; note that $k\leq \log |L|$. This proves the following theorem.

\begin{theorem}\label{thm:IsomSolv}
Let $L$ and $\tilde L$ be two solvable cube-free groups. Algorithm \ref{algo:lift-all} is a polynomial-time algorithm  that constructs an isomorphism $L\to \tilde L$, and reports false if and only if $L\not\cong\tilde L$. 
\end{theorem}

\section{Proof of Theorem~\ref{thm:main} (Isomorphism testing of cube-free groups)}\label{sec:general}
We now prove our main result,  Theorem~\ref{thm:main}, by describing Algorithm \ref{algo:general}. Recall from Section \ref{secStructure} that every cube-free group has 
the form $G=A\times L$, with $L$ solvable and $A=1$ or $A={\rm PSL}_2(p)$. If $A\ne 1$, 
then $A=G^{(3)}$, the third term of the derived series of $G$, see Remark \ref{remMeta}.
We compute $G^{(3)}$ using the normal closure of commutators \cite{Seress}*{p.\ 23}; 
since membership testing in permutation groups is in deterministic polynomial time, this can be done efficiently.  Furthermore, as $G^{(3)}$ is normal, the algorithm of \cite{KL:quo}*{P6}
applies to compute $L=C_G(A)$  in polynomial time. Thus, we may decompose $G=A\times L$, and 
likewise $\tilde{G}$, in polynomial time. For Step (1) of the general algorithm, the construction of an isomorphism $\psi_A\colon A\to\tilde A$, we use the next proposition. The correctness of Algorithm \ref{algo:general} now follows from Theorem \ref{thm:IsomSolv}; together with Proposition~\ref{prop:PSL}, the runtime is polynomial in the input size. 

\begin{proposition}\label{prop:PSL}
Let $A$ be  isomorphic to a non-abelian simple group of cube-free order. 
There is a polynomial-time algorithm that returns an isomorphism $A\to {\rm PSL}_2(p)$.
\end{proposition}
\begin{proof} 
 By assumption, $A\cong {\rm PSL}_2(p)$. We can determine $p$ by computing $|A|$, and then find $x,y\in A$  of order $p$ and $(p+1)/2$, respectively; note that $\langle x,y\rangle\cong \text{PSL}_2(p)$ since $x$ generates a Sylow $p$-subgroup, and $y$ generates the image in $\text{PSL}_2(p)$ of the $(p-1)$-th power of a  Singer cycle in $\text{GL}_2(p)$.  Now construct a presentation $\langle x,y\mid R\rangle$ for
$A$ from these elements. In ${\rm PSL}_2(p)$,  list all element pairs $(x',y')$  of order
$p$ and $(p+1)/2$, respectively, and search for an identification $x\mapsto x'$ and $y\mapsto y'$
that satisfies the relations $R$. Once found,  return
the result as the isomorphism. If ${\rm PSL}_2(p)$ is represented on $n$ points, then $p\leq n$ and hence $|{\rm PSL}_2(p)|\leq n^3$.  The algorithm searches $|{\rm PSL}_2(p)|^{2}\leq n^6$ pairs, so this brute-force test ends in time polynomial in the input. 
\end{proof}
 
Proposition~\ref{prop:PSL} is a shortcut, available because of our focus  on
a polynomial-time algorithm for permutation groups.  Recognizing $A\cong {\rm PSL}_2(p)$ and constructing an isomorphism has been a subject of intense research; a polynomial time solution for groups of black-box type is discussed in \cite{sukru}.

\section{Examples}\label{sec:ex}
We have implemented the critical features of our algorithm in \cite{cube-free}, and we give a few demonstrations
of its efficiency in Table 1. For each test, we constructed two  (non-)isomorphic groups: 
we usually started with direct products of groups provided by {\sf GAP}'s SmallGroup 
Library, and then created isomorphic random copies $G$ and $H$ of these groups 
(by using random polycyclic generating set). For some of the groups 
we have used, Table 1  gives their size and \emph{code}; this data can be 
used to reconstruct the groups via the {\sf GAP} function {\sf PcGroupCode}. We  
applied our function {\sf IsomorphismCubefreeGroups} to find an isomorphism $G\to H$. When comparing the efficiency of our implementation with the  {\sf GAP} function 
{\sf Isomorphism\-Groups}, we have started both calculations with freshly constructed 
groups $G$ and $H$, to make sure  that previously computed data is not stored. We note that {\sf GAP} also provides a randomized function ({\sf RandomIsomorphismTest}) that attempts to decide isomorphism between finite solvable groups (given via their size and code); 
the current implementation does not return  isomorphisms. That algorithm runs exceedingly 
fast on many examples, see Table~1, but its randomized approach means it cannot be
guaranteed to detect all isomorphisms. There are some practical bottlenecks in our implementation which currently applies available libraries for pc-groups (cf.\ Section~\ref{remPhi}) and matrix groups 
(cf.\ Section \ref{sec:solvFF}). The efficiency problems for \emph{collection} 
(cf.\ \cite{collection}) become visible when larger primes are involved. (This is one reason why it takes several minutes  
to reconstruct some of the groups in Table~1 via {\sf PcGroupCode}.) 
Moreover, {\sf GAP}'s functionality for matrix groups is not yet making full 
use of the promising advances of the \emph{matrix group recognition project}. These bottlenecks are responsible for the 
long runtime of the  examples involving the 
prime $12198421$, which is large from the perspective of {\sf GAP}.  Nevertheless, as a proof of concept, these  examples  demonstrate well the efficiency of our algorithm compared to existing  methods.

\pagebreak
 
\begin{algorithm}[!htbp] 
{\small 
\caption{{\sf CyclicLift}}
\label{algo:lift-cyclic}
\begin{algorithmic}[1] 
  \Require cube-free solvable groups $Y^\ast,\tilde Y^\ast$ with $|Y^\ast|=|\tilde Y^\ast|$, subgroups $A\leq\Phi(Y^\ast)$ and $\tilde A\leq \Phi(\tilde Y^\ast)$ isomorphic to $C_p$, natural projections $\Gamma\colon Y^\ast\to Y^\ast/A$ and $\tilde \Gamma\colon \tilde Y^\ast\to \tilde Y^\ast/\tilde A$ with images $Y=Y^\ast\Gamma$ and $\tilde Y=\tilde Y^\ast\tilde \Gamma$, and an isomorphism $\varphi\colon Y\to \tilde Y$
\Ensure an isomorphism $\hat\varphi\colon Y^\ast\to\tilde Y^\ast$

\vspace*{-2ex}

\begin{tabbing}
\hspace*{1em}\=\hspace*{1em}\=\hspace*{1em}\=\hspace*{1em}\=\hspace*{1em} \\
{\bf def } {\sf CyclicLift}($Y^\ast,A,\Gamma,\tilde Y^\ast, \tilde A,\tilde \Gamma,\varphi$)\\
\>\> use Proposition \ref{prop:Sylow-tower} to get a Sylow basis $Y_1^\ast,\ldots,Y_n^\ast$ of $Y^\ast$\\
\>\> define $H=\prod\nolimits_{k\ne i} Y_k^\ast$, where $A\leq Y_i^\ast$; this is a Hall $p'$-subgroup of $Y^\ast$\\
\>\> construct $\tilde H$ as a Hall $p'$-subgroup in the preimage of $H\Gamma\varphi$ under $\tilde\Gamma$\\
\>\> construct induced isomorphisms $\pi\colon H\to H\Gamma$ and $\tilde\pi\colon\tilde H\to \tilde H\tilde\Gamma$\\
\>\> fix a generator $a$ of $Y_i^\ast$ (Lemma \ref{lemFEF}) and let $\tilde a\in \tilde Y^\ast$ be a preimage  of $a\Gamma\varphi$ under $\tilde \Gamma$\\
\>\> let $M$ be a generating set of $H$\\
\>\> define $\hat\varphi\colon Y^\ast\to \tilde Y^\ast$ by mapping each $m\in M$ to $m\pi\varphi\tilde\pi^{-1}$, and $a$ to $\tilde a$\\
\>\> \Return $\hat\varphi$
\end{tabbing}
\end{algorithmic}}
\end{algorithm}

\begin{algorithm}[!htbp] 
{\small
\caption{\sf Lift}
\label{algo:lift-all}
\begin{algorithmic}[1] 
\Require cube-free solvable groups $L$ and $\tilde L$ of the same order
\Ensure an isomorphism $\hat\varphi\colon L\to\tilde L$, or  false if $L\not\cong \tilde L$

\vspace*{-2ex}

\begin{tabbing}
\hspace*{1em}\=\hspace*{1em}\=\hspace*{1em}\=\hspace*{1em}\=\hspace*{1em} \\
{\bf def } {\sf Lift}($L,\tilde L$)\\
\> compute $\Phi(L)$ and $\Phi(\tilde L)$, see Section \ref{remPhi}\\
\> {\bf if} $|\Phi(L)|=|\Phi(\tilde L)|=1$ {\bf then}\\
\>\> use the algorithm of Section \ref{secsolvFF} to get an isomorphism $\hat\varphi\colon L\to \tilde L$, or 
return false if that fails\\
\> {\bf else}\\
\>\> decompose $\Phi(L)=Y_{p_1}\times\ldots\times Y_{p_n}$ and $\Phi(\tilde L)=\tilde Y_{p_1}\times\ldots\times \tilde Y_{p_n}$ into Sylow subgroups\\
\>\> for each $i$ define $M_i=Y_{p_i}\times\ldots\times Y_{p_n}$ and  $\tilde M_i=\tilde Y_{p_i}\times\ldots\times \tilde Y_{p_n}$, with $M_j=1=\tilde M_j$ for $j>n$\\
\>\> for each $i$ define $L_i=L/M_i$ and $\tilde L_i=\tilde L/\tilde M_i$\\
\>\> for each $i\geq 2$ define natural projections $\pi_i\colon L_{i}\to L_{i-1}$ and $\tilde \pi_i\colon\tilde L_{i}\to \tilde L_{i-1}$\\
\>\> use the algorithm of Section \ref{secsolvFF} to get an isomorphism $\hat\varphi\colon L_1\to \tilde L_1$, or 
return false if that fails\\
\>\> {\bf for }{$i=2,\ldots,n+1$} {\bf do}\\
\>\>\> set $\hat\varphi=${\sf CyclicLift}$(L_{i}, M_{i-1}/M_{i}, \pi_i,
\tilde L_{i},\tilde M_{i-1}/\tilde M_{i},\tilde\pi_i,\hat\varphi)$, which is an isomorphism $L_i\to\tilde L_i$\\
\>\> \Return $\hat\varphi$
\end{tabbing}
\end{algorithmic}}
\end{algorithm}

\begin{algorithm}[!htbp]   
{\small
\caption{\sf IsomorphismCubefreeGroups}
\label{algo:general}
\begin{algorithmic}[1]
  \Require Cube-free groups $G$ and $\tilde G$ of the same order
  \Ensure an isomorphism $\varphi: G\to\tilde G$, or false if $G\not\cong \tilde G$

\vspace*{-2ex} 
  
\begin{tabbing}
\hspace*{1em}\=\hspace*{1em}\=\hspace*{1em}\=\hspace*{1em}\=\hspace*{1em} \\
{\bf def } {\sf IsomorphismCubefreeGroups}($G,\tilde G$)\\
\> compute $A=G^{(3)}$ and $L=C_G(A)$,  as well as $\tilde A=\tilde G^{(3)}$ and $\tilde L=C_{\tilde G}(\tilde A)$\\
\> construct an isomorphism $\psi_A\colon A\to \tilde A$, or return false if $A\not\cong \tilde A$,  see Proposition \ref{prop:PSL}\\
\> construct $\psi_L=${\sf Lift}$(L,\tilde L)$, which is an isomorphism $\psi_L\colon L\to\tilde L$, or return false if $L\not\cong \tilde L$\\
\> combine $\psi_A$ and $\psi_L$ to  an isomorphism $\varphi\colon G\to \tilde G$\\
\> {\bf return} $\varphi$.
\end{tabbing}
\end{algorithmic}}
\end{algorithm}

\begin{landscape}
\begin{table}
\hspace*{-4.4cm}\begin{tabular}{lp{15.9cm}}\label{tab1} 
  {\bf size}:& $213444=2^2.3^2.7^2.11^2$ (two isomorphic groups)\\
  \multicolumn{2}{l}{Runtime {\sf IsomorphismCubefreeGroups}: 0.12 seconds; {\sf GAP} runtimes: 110 seconds and 0.30 seconds}\\[5ex]

  {\bf size}:& $485100=2^2.3^2.5^2.7^2.11$  (two isomorphic groups)\\
   \multicolumn{2}{l}{Runtime {\sf IsomorphismCubefreeGroups}: 0.14 seconds; {\sf GAP} runtimes: 9.25 hours and 0.10 seconds}\\[5ex]

   {\bf size}:& $2455229080695145234788= 2^2. 3^2. 7.11.17.23.29^2.59.709.2837.22697$  (two isomorphic groups)\\
  \multicolumn{2}{l}{Runtime {\sf IsomorphismCubefreeGroups}: 46 seconds; {\sf GAP} runtimes: (aborted)  and 12.85 hours}\\[1ex]
 code:& {\footnotesize25771887290058268324444222548427618466622535561188418157206222315530817636985160639832764682223398558454926208711434863233254329561285}\\&{\footnotesize7310614599377329545076424741385533019060045922880910282042489387835906289279581907750184052068613887290089849139978833781413618189}\\[1ex]
 code:& {\footnotesize42935964225237064245986914596365100273683956747676598979814176980191348433158368824756791059830426394631361311711333822038779784490919}\\&{\footnotesize1398533193638418225692067093120389360092220226273076405684036236511208423558471856377830123474389120161517062590458151937327292273539}\\[5ex]

 {\bf size}:& $148801462694820=2^2.3^2.5.13^2.401.12198421$  (two isomorphic groups)\\
     \multicolumn{2}{l}{Runtime {\sf IsomorphismCubefreeGroups}: 1.34 hours; {\sf GAP} runtimes: (aborted)  and 43.02 hours}\\[1ex]
   code:& {\footnotesize33485470139896255235932843080490226656789884890216293350628774303597330608696148032177992911339898019268212938339562678223839825646765}\\&{\footnotesize36143476701568788444140906714234850667635698469843932592713738130822523580315216756068451815666063208366321490271081072566186700588774}\\&{\footnotesize041361401470419}\\[1ex]
   code:&{\footnotesize30847018874524812119899977696213501488384865133706306032190042564646975355564996929067797274360859513233682334426380109067389439911324}\\&{\footnotesize7457353858208706985108929792505014584900163986344080431816556839}\\[5ex]
   
   {\bf size}:& $11 793 441 660=2^2.3.5.7.11^2.13.17851$  (two non-isomorphic groups)\\
    \multicolumn{2}{l}{Runtime {\sf IsomorphismCubefreeGroups}: 1.00 seconds; {\sf GAP} runtimes: (aborted)  and (not applicable)}\\[1ex]
  code:&{\footnotesize130759863164212785921829892045963491290671156934582787304559199096594506157779256403437080197441699}\\[1ex]
  code:&{\footnotesize140818314176844538685084602259218084360152269575198837414544118817517722796914042489570754478966750}\\[3.5ex]   
  \end{tabular}    
  \caption{Comparison of runtimes of isomorphism tests for some cube-free groups; {\sf GAP} runtimes are given for the {\sf GAP} functions {\sf IsomorphismGroups} and {\sf RandomIsomorphismTest} (in that order); we aborted computations which used more that 20GB of memory}    
\end{table}
\end{landscape}

{\small
\bibliographystyle{abbrv}

}

\end{document}